\documentclass[oneside]{amsart}
\usepackage[utf8]{inputenc}
\usepackage{amsmath,amsthm,amssymb}
\usepackage{mathtools}
\usepackage{mathrsfs} 
\usepackage{graphicx, xcolor}
\usepackage{overpic}
\usepackage{caption} 
\usepackage{subfiles}
\usepackage{newverbs} 
\usepackage{array} 
\usepackage{tikz-cd}
\usepackage{braket}
\usepackage{enumitem}

\usepackage[colorlinks,citecolor=blue,linkcolor=red!70!black]{hyperref}
\usepackage[nameinlink]{cleveref}

\usepackage[marginpar=4.4cm,left=5cm,right=5cm]{geometry}

\usepackage{amssymb} 
\usepackage{graphicx}

\newverbcommand{\CMRverb}{\tiny\color{blue}}{}
\newverbcommand{\GRverb}{\tiny\color{teal}}{}
\newverbcommand{\STverb}{\tiny\color{cyan}}{}
\newverbcommand{\Pverb}{\tiny\color{violet}}{}
\newverbcommand{\Averb}{\tiny\color{brown}}{}

\theoremstyle{plain}
\newtheorem{THM}{Theorem}[section]

\newtheorem{PROP}[THM]{Proposition}

\newtheorem{FACT}[THM]{Fact}

\theoremstyle{definition}
\newtheorem{DEF}[THM]{Definition}
\newtheorem{RMK}[THM]{Remark}
\newtheorem{EX}[THM]{Example}

\newtheoremstyle{named}{}{}{\itshape}{}{\bfseries}{.}{.5em}{\thmnote{#3}}
\theoremstyle{named}
\newtheorem*{namedtheorem}{Theorem}

\DeclareMathOperator{\id}{Id}

\newcommand{\dfn}{\mathrel{\mathop:}=}

\newcommand{\C}{\mathcal{C}}

\newcommand{\R}{\mathbb{R}}
\renewcommand{\S}{\Sigma}
\newcommand{\Z}{\mathbb{Z}}

\newcommand{\V}{\mathcal{V}}

\newcommand{\PMap}{\operatorname{PMap}}
\newcommand{\Mp}{\operatorname{Map}}
\newcommand{\Homeo}{\operatorname{Homeo}}

\let\oldphi\phi 
\let\phi\varphi 
\let\varphi\oldphi

\title[CB Classification of $\PMap(\S)$]{Large-Scale Geometry of Pure Mapping Class Groups of Infinite-Type Surfaces}
\author[Hill]{Thomas Hill}

\begin{document}
\maketitle

\begin{abstract}
    
    The work of Mann and Rafi gives a classification surfaces $\S$ when $\Mp(\S)$ is globally CB, locally CB, and CB generated under the technical assumption of \emph{tameness}.  In this article, we restrict our study to the pure mapping class group and give a complete classification without additional assumptions.  In stark contrast with the rich class of examples of Mann--Rafi, we prove that $\PMap(\S)$ is globally CB if and only if $\S$ is the Loch Ness monster surface, and locally CB or CB generated if and only if $\S$ has finitely many ends and is not a Loch Ness monster surface with (nonzero) punctures.  
\end{abstract}

\section{Introduction}

The mapping class group of a surface $\S$ is the group of orientation-preserving homeomorphisms up to isotopy and is denoted by $\Mp(\S)$.  The subgroup that fixes the ends of the surface pointwise is called the pure mapping class group and is denoted by $\PMap(\S)$.  These groups form a short exact sequence with the set of homeomorphisms of the end space:

 \begin{equation*} 
    1 \to \PMap(\S) \to \Mp(\S) \to \Homeo(E(\S), E_G(\S)) \to 1 
    \end{equation*} 
    
From this sequence, we see that understanding $\Mp(\S)$ involves understanding two components:  $\Homeo(E(\S), E_G(\S))$ thought of as the group of ``external symmetries'' of the surface, and $\PMap(\S)$ the group of ``internal symmetries''.  
The set of ends is homeomorphic to a subset of a Cantor set, rendering $\Homeo(E(\S), E_G(\S))$ a highly complex object.
On the other hand, there are many well-established results concerning pure mapping class groups of infinite-type surfaces (see \cite{aramayona2020first},\cite{domat2022big},\cite{patel2018algebraic}).

This article aims to describe the large-scale geometry of pure mapping class groups of infinite-type surfaces. 
Mapping class groups of \emph{finite}-type surfaces have straightforward large-scale geometry. 
The Dehn-Lickorish theorem  \cite[see Chapter 4]{primer} establishes that the (pure) mapping class group of a finite-type surface is finitely generated. 
This finite generating set induces a word metric on $\Mp(\S)$. 
Similarly, a different set of generators induces a (possibly) different word metric on $\Mp(\S)$.
The finiteness of both generating sets allows one to establish a quasi-isometry between these two metric spaces, thereby showing that $\Mp(\S)$ possesses a well-defined \emph{quasi-isometry type}.

Characterizing the large-scale geometry of big mapping class groups, specifically identifying whether they possess a well-defined quasi-isometry type and determining that type, presents a significantly greater challenge compared to mapping class groups of surfaces with finite type. 
This disparity arises because the conventional techniques of geometric group theory, applicable to finitely generated groups, extend to locally compact, compactly generated topological groups, but big mapping class groups do not fall into this category.
Fortunately, Rosendal's work \cite{rosendal2021coarse} provides a framework to extend the tools of geometric group theory to big mapping class groups and, more generally, to Polish groups. This framework replaces the compactness conditions with the notion of \emph{locally coarsely bounded} (abbreviated locally CB) and \emph{CB-generated} (i.e., generated by a coarsely bounded set).  

Under the additional assumption of tameness, Mann and Rafi described the large-scale geometry of the full mapping class group by classifying which surfaces have a CB, locally CB, or CB-generated mapping class group \cite{mann2019large}. Within their classification, there are infinitely many surfaces that have a CB mapping class group, and an even larger class of surfaces have a locally CB and CB-generated mapping class group.  In contrast, we show that a relatively small class of surfaces have a CB, locally CB, or CB-generated \emph{pure} mapping class group.   Furthermore, by restricting to the pure mapping class group, we give a CB classification of $\PMap(\S)$ without the additional tameness condition: 

\begin{THM}
\label{thm:mainthm}
Let $\S$ be an infinite-type surface.  Then $\PMap(\S)$ is 
\begin{itemize}
    \item[\emph{(a)}] globally CB if and only if $\S$ is the Loch Ness monster surface;
    \item[\emph{(b)}] locally CB if and only if $|E(\S)| < \infty$ and $\S$ is not a Loch Ness monster surface with (nonzero) punctures;
    \item[\emph{(c)}] CB-generated if and only if $\PMap(\S)$ is locally CB. 
\end{itemize}
\end{THM}

Many of the results of Mann--Rafi restrict nicely to the pure mapping class group and serve as tools to prove \Cref{thm:mainthm}, including using nondisplaceable subsurfaces to disprove the CB-ness of the mapping class group.  However, a fundamental difference in the setting of $\PMap(\S)$ is that a surface with more than three (and possibly infinitely many) ends always contains a nondisplaceable subsurface:

\begin{PROP}
    \label{lem:threeEnds}
    Let $\S$ be an infinite-type surface, possibly with an infinite end space. If $|E(\S)| \ge 3$ then $\S$ contains a nondisplaceable subsurface with respect to $\PMap(\S).$
\end{PROP}      

This proposition drastically reduces the possible candidates of surfaces with a CB pure mapping class group, and the remaining possibilities are considered case-by-case.  \Cref{lem:threeEnds} is also an important new tool used to classify locally CB and CB-generated pure mapping class groups.

\subsection*{Acknowledgements}
The author extends appreciation to Priyam Patel for her insightful project suggestion, invaluable feedback, and mentorship. The author also acknowledges Sanghoon Kwak for his extensive and insightful conversations, along with his meticulous comments and feedback on multiple drafts of the paper. Appreciation also goes to Mladen Bestvina and Rebecca Rechkin for their helpful discussions that significantly contributed to the development of this project.  The author acknowledges support from RTG DMS–1840190.

\section{Preliminaries}

\subsection{Surfaces and Mapping Class Groups}

Surfaces are 2-dimensional manifolds, and they come in two flavors: finite-type surfaces and infinite-type surfaces.  
\textbf{Finite-type surfaces} 
are those with a finitely generated fundamental group. 
These surfaces are classified by their genus $g$, number of boundary components $b$, and number of punctures $n$.   
On the other hand, those that do not have a finitely generated fundamental group are called \textbf{infinite-type surfaces}.  For infinite-type surfaces, the triple $(g, b, n)$ does not contain enough information to distinguish between distinct surfaces. For example, the Loch Ness monster surface and the ladder surface both have $(g, b, n) = (\infty, 0, 0)$, see \Cref{pic:ladder_LN}.  
The classification of infinite-type surfaces also depends on the \emph{end space} of the surface.  Roughly speaking, the end space represents the distinct ways to move toward infinity on the surface.  

\begin{figure}[ht!]
    \centering
    \includegraphics[scale=.15]{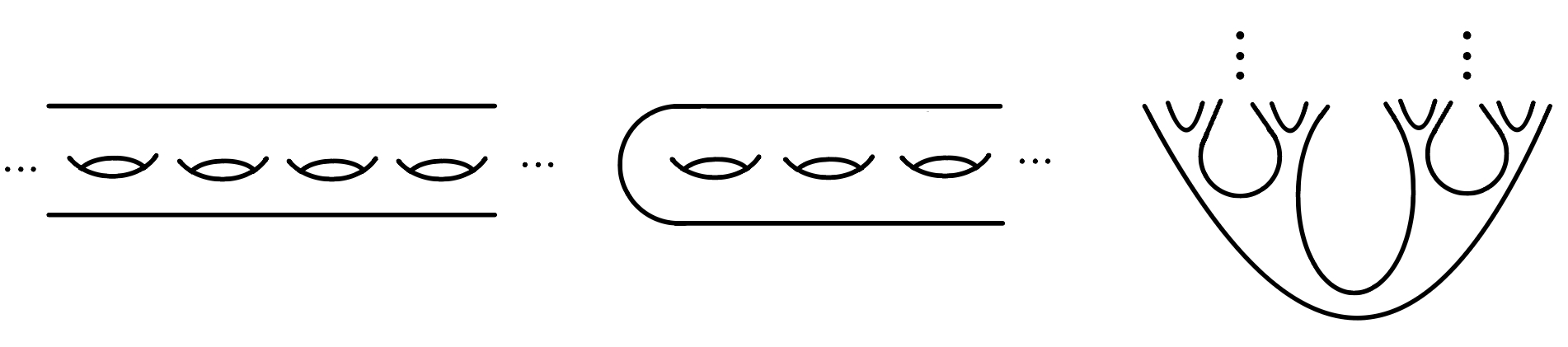}
    \caption{The ladder surface (left) has two distinct ends accumulated by genus, while the Loch Ness monster surface (center) has only one. The Cantor tree surface (right) is a sphere with a Cantor set removed and has uncountably many planar ends.  
    \label{pic:ladder_LN}
    }
\end{figure}

Put precisely, an \textbf{end} is an equivalence class of nested connected subsurfaces of $S$, denoted by $[U_1 \supset U_2 \supset \cdots ]$, such that each $U_i$ has a compact boundary, and any compact subsurface $K\subset S$ is disjoint from $U_n$ for large enough $n$.  Two sequences $U_1 \supset U_2 \supset \cdots$ and $V_1 \supset V_2 \supset \cdots$ are equivalent if they are both eventually contained in one another; that is, if for any $n$ there is an $m$ such that $U_m \subset V_n$ and $V_m \subset U_n$.  
An end $[U_1 \supset U_2 \supset \cdots ]$ is called \textbf{planar} if the $U_i$ are all eventually homeomorphic to an open subset of $\R^2$.  
If instead all $U_i$ in the sequence have infinite genus, the end is said to be \textbf{accumulated by genus} or \textbf{non-planar}.  
The collection of all ends is a topological space called the \textbf{end space} of $\S$ and is denoted by $E(\S)$. The subset of ends accumulated by genus is denoted by $E_G(\S)$.
As a topological space, the points of $E(\S)$ are ends $[U_1 \supset U_2 \supset \cdots]$, and open sets of $E(\S)$ can be thought of as subsurfaces of $\S$ with compact boundary. 
To be precise, the collection of all ends eventually contained in a subsurface $U\subset \S$
with compact boundary, forms a basis for the topology of $E(\S)$.
An end $[U_1 \supset U_2 \supset \cdots]$ is said to be contained in the open set $U$ if $U_i \subset U$ for large enough $i$. 
With the additional data of the end space, infinite-type surfaces can be completely classified by the homeomorphism type of the tuple $(g, b, E(\S), E_G(\S))$, see \cite{Richards1963OnTC}. We will only consider surfaces with empty boundary.

The group of orientation-preserving homeomorphisms up to isotopy is called the \textbf{mapping class group} and is denoted by $\Mp(\S)$. The subgroup of $\Mp(\S)$ that restricts to the identity on $E(\S)$ is called the \textbf{pure mapping class group} and is denoted by $\PMap(\S)$.    Mapping class groups of infinite-type surfaces are often called \emph{big} mapping class groups. 

\subsection{Large-Scale Geometry and Coarse Boundedness} Describing the large-scale geometry, or well-defined quasi-isometry type, of big mapping class groups is more challenging than for finite-type surfaces, as standard geometric group theory tools for finitely generated groups do not directly apply to big mapping class groups.  Rosendal's work \cite{rosendal2021coarse} introduces the notion of \emph{locally coarsely bounded} and \emph{CB-generated} groups, providing a framework to extend the tools of geometric group theory to a larger class of groups.  This subsection aims to summarize Rosendal's key contributions.

\begin{DEF} 
    \label{def:globallyCB} A subset $A$ of a topological group $G$ is \textbf{globally coarsely bounded} in $G$, or simply \textbf{CB}, if for every continuous isometric action of $G$ on a given metric space $X$, the orbit $A \cdot x$ is bounded for all $x \in X$.  
\end{DEF}

 A topological group is \textbf{Polish} if it is separable and completely metrizable. Both $\Mp(\S)$ and $\PMap(\S)$ are examples of Polish groups \cite{aramayona2020big}. For Polish groups, Rosendal provides a useful alternative criterion for coarse boundedness: 

\begin{FACT}[{\cite[Prop. 2.7 (5)]{rosendal2021coarse}}]
    \label{fact:RosCrit} A subset $A$ of a Polish group $G$ is CB if and only if for every neighborhood of the identity $\V$ there is a finite subset $\mathcal{F} \subset G$ and $k \ge 1$, so that $A \subset (\mathcal{F}\V)^k$.  
\end{FACT}

In the remark below, we use \Cref{fact:RosCrit} to show that length functions on CB Polish groups are always bounded.  This is key in the proof of \Cref{prop:nondisp}, where the existence of an unbounded length function is used to disprove CB-ness. 
Recall that a \textbf{length function} on a group $G$ is a continuous map $\ell \colon G \to [0,\infty)$ that satisfies $\ell(\id) = 0$, $\ell(g) = \ell(g^{-1})$, and $\ell(gh)\le \ell(g) + \ell(h)$ for all $g,h \in G$.  

\begin{RMK} 
\label{ex:LenCB}
If $G$ is a CB Polish group, then any length function $\ell \colon G \to [0, \infty)$ is bounded. Indeed, take $\V = B_r(\id)$, then by CB-ness of $G$ there is a finite set $\mathcal{F}$ and $k \ge 0$ so that $G \subset (\mathcal{F} \cdot B_r(\id))^k$.  This means $$\displaystyle{\ell(\phi) \le \left( \max_{f \in \mathcal{F}} \ell(f) + r\right)\cdot k},$$ for all $\phi \in G$. 
\end{RMK}

\begin{DEF}
    \label{def:locallyCB}
    A group is \textbf{locally CB} if it has a CB neighborhood of the identity.  
\end{DEF}

When discussing the (local) CB-ness of a given set, it is essential to keep track of the ambient group since it is possible that $H \le G$ might be CB in $G$, but $H$ might not be CB as a subset of itself. For instance, consider the following example.

\begin{EX}
    \label{ex:ambientCBness}
    Let $\S$ be a surface for which $\Mp(\S)$ is CB, and let $\gamma$ be a non-separating simple closed curve in $\S$.  Consider the group generated by a Dehn twist about $\gamma$, namely $\langle T_\gamma \rangle \le \Mp(\S)$. 
    Because $\Mp(\S)$ is assumed to be CB, any subset $A \subseteq \Mp(\S)$, and in particular $\langle T_\gamma \rangle$ must also be CB in $\Mp(\S)$. 
    Thus, $\langle T_\gamma \rangle$ is CB in $\Mp(\S)$, but $\langle T_\gamma \rangle$ is not CB as a subset of itself since $\langle T_\gamma \rangle \cong \mathbb{Z}$ and the continuous isometric action of $\Z$ on $\R$ by translations has unbounded orbits. 
\end{EX}

When $G$ is a Polish group and $H \le G$ is a \emph{finite index} open Polish subgroup, this distinction can be blurred in light of the following fact (see {\cite[Proposition 2.20]{domat2023coarse} } for a proof).  

\begin{FACT}[{\cite[Proposition 5.67]{rosendal2021coarse}}] 
\label{fact:CEmbed}
    Let $G$ be a Polish group and $H \le G$ a finite index open Polish subgroup.  Then $H$ is CB if and only if $G$ is CB. 
    Furthermore, $H$ is coarsely embedded in $G$; that is to say, for $A \subset H$,  $A$ is CB in $H$ if and only if $A$ is CB in $G$.   
    \end{FACT} 

\begin{DEF}
    \label{def:CBgen}
    A group is \textbf{CB-generated} if it admits a CB generating set. 
\end{DEF}

The relationship between CB, locally CB, and CB-generated Polish groups can be summarized in the diagram of implications below: 
\begin{center}
    \begin{tikzcd}
        \operatorname{CB} \arrow[rd, Rightarrow,"(3)",swap] \arrow[r, Rightarrow,"(1)"] & \operatorname{CB-generated} \arrow[d, Rightarrow,"(2)"] \\
            & \text{\rm locally CB}        
    \end{tikzcd}
\end{center}

When a group is CB, it is CB-generated since the whole group can be taken as the CB-generating set, so implication (1) holds. Because CB-generated groups have a
well-defined quasi-isometry type by \cite{rosendal2021coarse},  this means in particular that globally CB groups have trivial quasi-isometry type.  
Implication (2) is a result of Rosendal that every CB-generated Polish group is locally CB {\cite[Theorem 1.2]{rosendal2021coarse}}.

Of course, combining (1) and (2) results in implication (3); however, (3) can also be seen directly. If a group is globally CB, then it is locally CB since the whole
group can be taken as the CB neighborhood of the identity.

\subsection{CB Classification of $\Mp(\S)$}

This section will briefly summarize some of the results and tools of Mann--Rafi's CB classification of full mapping class groups \cite{mann2019large}.  

\begin{THM}[{\cite[Theorem 1.7]{mann2019large}}] 
\label{fact:MRgloballyCB}
Suppose that $\S$ is either tame or has a countable end space.  Then $\Mp(\S)$ is CB if and only if $\S$ has infinite or zero genus and $E(\S)$ is self-similar or telescoping.  
\end{THM}

An end space $E(\S)$ is self-similar if for any decomposition $E(\S) = E_1 \sqcup E_2 \sqcup \dots \sqcup E_n$ there is a clopen subset $D \subset E_i$ for some $i$, such that $(D, D\cap E_G(\S)) \cong (E(\S), E_G(\S))$.  
Telescoping is a variation of the self-similar condition (see \cite[Definition 3.3]{mann2019large}).

\begin{RMK}
    \label{rmk:SSTelEndsInf}
    A telescoping surface necessarily has infinitely many ends (see {\cite[Proposition 3.7]{mann2019large}}).
    Likewise, self-similar end spaces contain either a single end or infinitely many ends.  
\end{RMK}

The definition of tameness space mentioned above in \Cref{fact:MRgloballyCB}
is a technical condition about the `local self-similarity' of certain ends of a surface.  We avoid discussing this condition since our main focus will be on surfaces with a finite end space, and tameness will always be satisfied for those surfaces. 

\begin{RMK}
    \label{rmk:finTame}
    If $|E(\S)| < \infty$, then $\S$ is tame. 
\end{RMK}

\begin{THM}[{\cite[Theorem 5.7]{mann2019large}} ] 
\label{fact:MRlocallyCB}
$\Mp(\S)$ is locally CB if and only if there is a finite-type subsurface $K \subset \S$ (possibly empty) such that the connected components of $\S \setminus K $ have infinite-type and either 0 or infinite genus, and partition $E(\S)$ into finitely many clopen sets
\begin{equation*} 
E(\S) = \left( \bigsqcup_{A \in \mathcal{A} } A\right) \sqcup \left( \bigsqcup_{P \in \mathcal{P}} P \right)
\end{equation*} 
such that: 
\begin{enumerate}
    \item Each $A \in \mathcal{A}$ is self-similar, $\mathcal{M}(A) \subset \mathcal{M}(E)$, and $\mathcal{M}(E) = \bigsqcup_{A \in \mathcal{A}} \mathcal{M}(A)$.
    \item Each $P \in \mathcal{P}$ is homeomorphic to a clopen subset of some $A \in \mathcal{A}$.
    \item For any $x_A \in \mathcal{M}(A)$, and any neighborhood $V$ of the end $x_A$ in $\S$, there is $f_V \in \Mp(\S)$ so that $f_V(V)$ contains the complementary component to $K$ with end space $A$.
\end{enumerate}
    Moreover, the set of mapping classes restricting to the identity on $K$, denoted by $\V_K$, is a CB neighborhood of the identity.  Additionally, $K$ can always be taken to have genus 0 if $\S$ has infinite genus and genus equal to $\S$ otherwise. If $\S$ has finitely many isolated planar ends, these may be taken as punctures of $K$.  
    
\end{THM}

\begin{EX}
    The ladder surface is an example of a surface with a locally CB, but not a globally CB mapping class group. Because it has finitely many ends, the ladder is a tame surface by \Cref{rmk:finTame}.  
    Since it has only two ends, the ladder surface is not self-similar or telescoping by \Cref{rmk:SSTelEndsInf}.  
    However, the ladder surface is locally CB since it satisfies the modified telescoping condition in \Cref{fact:MRlocallyCB} (3).  
\end{EX} 

\begin{THM}[{\cite[Theorem 1.6]{mann2019large}} ] 
    \label{fact:MRCBgen} 
    For a tame surface with a locally (but not globally) CB mapping class group, $\Mp(\S)$ is CB generated if and only if $E(\S)$ is finite-rank and not of limit type. 
\end{THM}
    
See {\cite[Definition 6.2 and 6.5]{mann2019large}}  for the definitions of finite rank and limit type, respectively. For our purpose, it is relevant that both the infinite rank and limit type conditions require the space to have an infinite number of ends.

\subsection{Nondisplaceable Subsurfaces}

For two subsurfaces $S$ and $S'$, we say $S \cap S' \ne \emptyset$ if every subsurface homotopic to $S$ intersects every subsurface homotopic to $S'$.  

\begin{DEF}
    A connected finite-type subsurface $S \subset \S$ is called \textbf{nondisplaceable} with respect to the group $G \le \Mp(\S) $ if $g(S) \cap S \ne \emptyset$ for all $g \in G$. 
\end{DEF}

Using the presence of a nondisplaceable subsurface in $\S$ to conclude that $\Mp(\S)$ is not CB is one of the key tools of Mann--Rafi. This is also important in disproving the CB-ness of $\PMap(\S)$ below.

\begin{PROP}[cf. {\cite[Theorem 1.9]{mann2019large}}]
    \label{prop:nondisp}
    If $\S$ has a nondisplaceable subsurface with respect to $\PMap(\S)$, then $\PMap(\S)$ is not globally CB. 
\end{PROP} 

The proof of \Cref{prop:nondisp} follows exactly as in Mann--Rafi, but with $\PMap(\S)$ instead of $\Mp(\S)$.  Below, we will give a proof sketch of their argument and justify the replacement of $\Mp(\S)$ with $\PMap(\S)$.  
The proof aims to exploit the unboundedness of the action of pseudo-Anosov elements on the curve complex of a finite-type subsurface proven by Masur--Minsky \cite{MR1714338}.  Curve graphs and subsurface projections will play a role in the proof.

    The \textbf{curve graph} of a surface $\S$, denoted by $\C(\S)$ is the graph whose vertices correspond with the isotopy classes of simple closed curves in $\S$ where two vertices $\alpha$ and $\beta$ are connected by an edge if $\alpha$ and $\beta$ can be realized disjointly.  The curve graph can be equipped with the path metric $d_{\S}$ induced by giving each edge length 1.  

    Let $S \subset \S$ be a finite-type subsurface.  For a curve $\gamma \in \C(\S)$, we define the \textbf{subsurface projection} of $\gamma$ onto $S$ as  
    \begin{center} 
    $\pi_S(\gamma) \dfn 
    \begin{cases}
    \\
    \\
    \\
    \end{cases}$ \hspace{-.75cm}
    \begin{tabular}{cc}
    $\emptyset$ & \text{ if there is a representative of $\gamma$ disjoint from $S$},\\
    $\gamma$ & \text{ if there is a representative of $\gamma$ contained in $S$,} \\
    \multicolumn{2}{l}{\text{
    surger $\gamma$ along $\partial S$ otherwise.}}
    \end{tabular}
    \end{center} 
     By surgering $\gamma$ along $\partial S$, we mean the following. If $\gamma \cap S \ne \emptyset$ but $\gamma \not\subset S$ then $\gamma \cap S$ is a union of arcs $\set{\alpha_i}$ based in $\partial S$.  Consider the boundary of a regular neighborhood of an arc $\alpha_i$.  The result is two arcs that we surger with the boundary of a regular neighborhood of the boundary component(s) of $S$ that $\alpha_i$ is based at.  The projection of $\pi_S(\gamma)$ is the union of all curves obtained from this surgery. See  \Cref{pic:subsurfproj} for an example of this surgery. 

    \begin{figure}[ht!]
    
    \centering
    \scalebox{.733333}{
    \begin{overpic}[width=9cm]{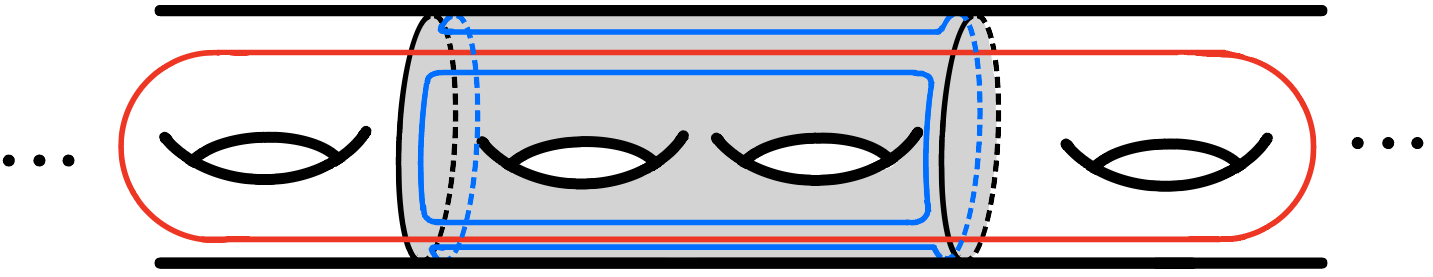}
    \put(11,12){\color{red}$\gamma$}
    \put(24.5,5){\color{black} $S$}
    \put(46,11){\color{blue} $\pi_S(\gamma)$}
    \end{overpic}
    }
    
    \caption{An example of the surgery used to define the subsurface projection. Here, $\pi_S(\gamma)$ is the disjoint union of the blue curves. 
    \label{pic:subsurfproj}
    }
\end{figure}
Similarly, if $\mu \subset \C(\S)$ is a multicurve, we define $\pi_S(\mu) \dfn \bigsqcup_{\gamma \in \mu} \pi_S(\gamma)$.  For a subsurface $R\subset \S$, we define $\pi_S(R) \dfn \pi_S(\partial R)$.  

The path metric $d_S$ on the curve complex $\C(S)$ of a finite-type subsurface $S \subset \S$ can be extended to curves $\gamma_1, \gamma_2 \in \C(\S)$ by defining 
\begin{equation*}
    d_S(\gamma_1, \gamma_2) \dfn \max_{\alpha_i \in \pi_S(\gamma_i)} d_S(\alpha_1, \alpha_2)
\end{equation*} 
and similarly for multicurves $\mu_1, \mu_2 \subset \C(\S)$, 
    $d_S(\mu_1, \mu_2) \dfn \displaystyle{\max_{\gamma_i \in \pi_S(\mu_i)} d_S(\gamma_1, \gamma_2)}.$

\noindent\emph{Proof of \Cref{prop:nondisp}.}  
The idea of Mann--Rafi's argument is to construct an unbounded length function $\ell \colon \Mp(\S) \to \Z$, which in light of  \Cref{ex:LenCB} is enough to show that $\Mp(\S)$ is not CB.  

We now show that their argument goes through with $\PMap(\S)$ in place of $\Mp(\S)$.  Let $S \subset \S$ be a nondisplaceable subsurface and let $\mathcal{I} = \set{f(S) \mid f \in \PMap(\S)}$. 
Let $\mu_S$ be a multi-curve that fills $S$.  Because $S$ is nondisplaceable, $R \cap S \ne \emptyset$ for all $R \in \mathcal{I}$, so $\mu_R \dfn \pi_R(\mu_S)$ is nonempty for all $R \in \mathcal{I}$.  Define a length function 
\begin{equation*}
    \ell \colon \PMap(\S) \to \Z^+ \qquad \text{given by} \qquad  \ell(\phi) \dfn \max_{R \in \mathcal{I}} d_{\phi(R)}(\phi(\mu_R), \mu_{\phi(R)})
\end{equation*} 
where $d_{\phi(R)}$ is the path metric in the curve complex $C(\phi(R))$.  
Intuitively, the function $\ell$ is measuring the maximum distance in the curve complex of $\phi(R)$ between curves in $\mu_R$ transformed by $\phi$ and curves in the intersection of $S$ and $\phi(R)$.

Mann--Rafi prove $\ell$ is a continuous length function and demonstrate that $\ell$ is unbounded by taking a $\phi \in \PMap(\S)$ that preserves $S$ and restricts to a pseudo-Anosov on $S$. Then 
\begin{equation*}
\ell(\phi^n) \ge d_{\phi^n(S)} (\phi^n(\mu_S), \mu_{\phi^n(S)}) = d_S(\phi^n(\mu_S), \mu_S) \ge d_S(\phi^n(\gamma), \gamma)
\end{equation*}
where $\gamma$ is any essential curve in $S$, and the middle equality holds because $\phi$ preserves $S$.  By the work of Masur--Minsky \cite{MR1714338} $d_S(\phi^n(\gamma), \gamma) \to \infty$ as $n \to \infty$ for such a map $\phi$, so $\ell$ is unbounded.\hfill $\square$

\section{CB Classification of $\PMap(\S)$}

\subsection{Globally CB Classification}

We show below that one of the striking differences between $\PMap(\S)$ and $\Mp(\S)$ is that surfaces with at least three ends contain a nondisplaceable subsurface with respect to $\PMap(\S)$. In contrast, there are many examples of surfaces with infinitely many ends that don't have a nondisplaceable subsurface with respect to $\Mp(\S)$.  For example, the Cantor tree surface pictured in  \Cref{pic:ladder_LN} has a globally CB mapping class group, and therefore, it does not have any nondisplaceable subsurfaces with respect to the (full) mapping class group.  
     
\begin{namedtheorem}[\Cref{lem:threeEnds}]
     Let $\S$ be an infinite-type surface, possibly with an infinite end space. If $|E(\S)| \ge 3$ then $\S$ contains a nondisplaceable subsurface with respect to $\PMap(\S).$
\end{namedtheorem}

\begin{proof}

First, we claim that there is a compact subsurface $K \subset \S$ whose complementary components partition $E(\S)$ into three nonempty disjoint subsets.  
To see why, suppose that $x_1, x_2, x_3 \in E(\S)$ are distinct ends.  Recall that by definition, $x_1$ is an equivalence class of nested connected subsurfaces with compact boundary $x_1 = [S_1^0 \supset S_1^1 \supset S_1^2 \supset \cdots]$. 
Because the ends $x_1, x_2, x_3$ are distinct, eventually there is a subsurface, say $S_1^i$, in the sequence $S_1^0 \supset S_1^1 \supset S_1^2 \supset \cdots$ that does not contain $x_2$ or $x_3$. 
By definition, $S_1^i$ is a connected surface with compact boundary.

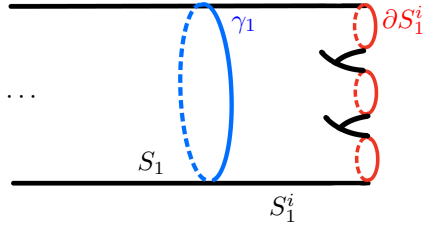
\begin{figure}[ht!]
    \centering
    \begin{overpic}[width=5cm]{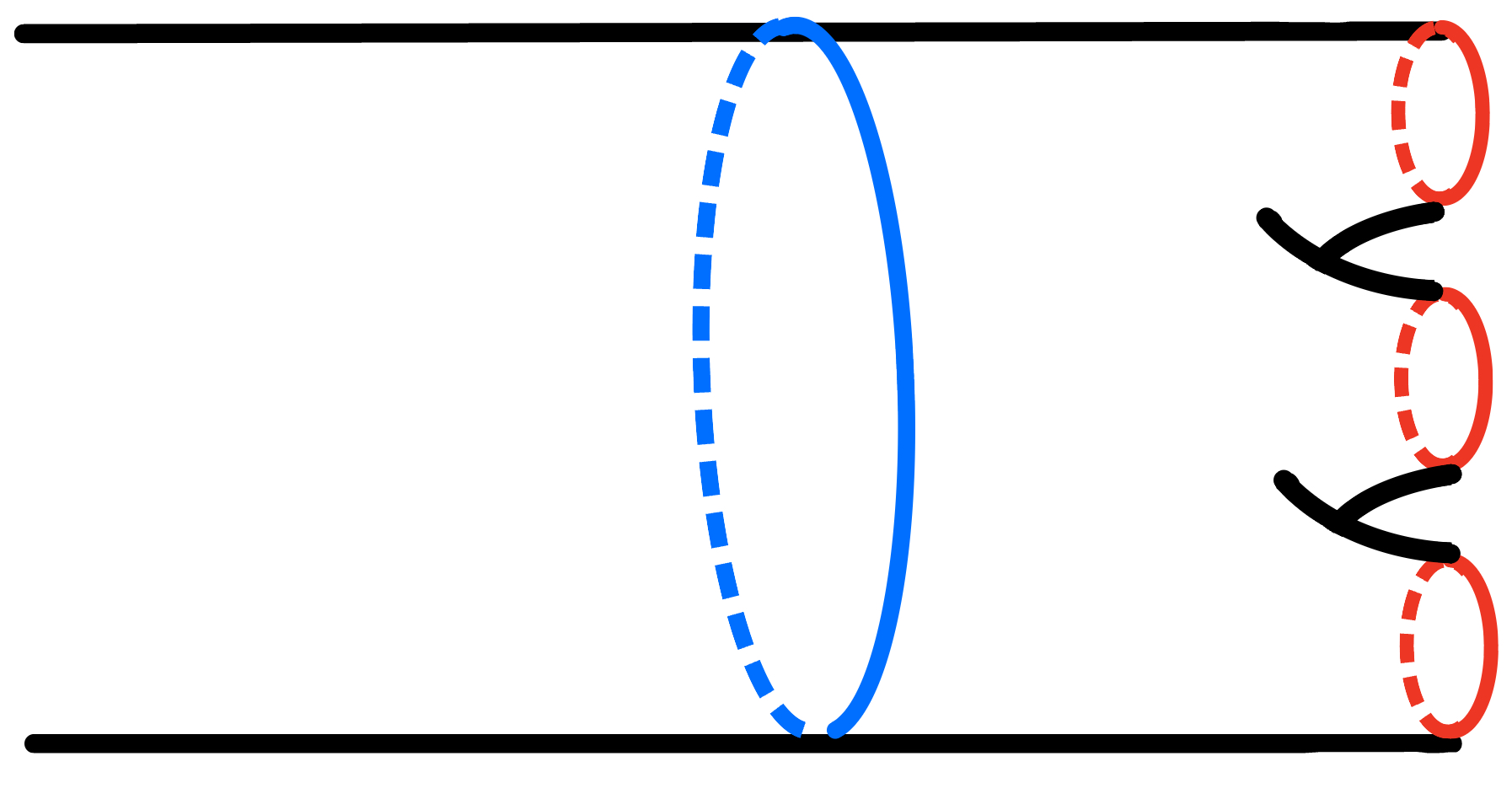}
        \put(60,45){\color{blue} $\gamma_1$}
        \put(35,7){$S_1$}
        \put(70,-5){$S_1^i$}
        \put(0,25){$\cdots$}
        \put(100,45){\color{red} $\partial S_1^i$}
    \end{overpic}
    \caption{Since $S_1^i$ is connected and has compact boundary, it can be visualized with the boundary components lined up as above, up to homeomorphism.  The separating curve $\gamma_1$ is chosen to cut off only the boundary components of $S_1^i$; in particular, $E(S_1^i) = E(S_1)$. 
    \label{pic:cptBdry}
    }
\end{figure}

As pictured in \Cref{pic:cptBdry}, we can find a separating simple closed curve $\gamma_1$, which separates all the boundary components of $S_1^i$ from the rest of the subsurface $S_1^i$.  
Cutting the original surface $\S$ along $\gamma_1$, we obtain two new subsurface, each with a single boundary component $\gamma_1$: one containing the boundary of $S_1^i$ call it $\S'$, and another subsurface, call it $S_1$, that has $E(S_1) = E(S_1^i)$.   
Thus the curve $\gamma_1$ partitions the end space of $\S$ into two subsets $E(\S) =  E(S_1) \sqcup E(\S')$, where $x_1 \in E(S_1)$ and $x_2, x_3 \in E(\S')$.

We can apply the same strategy to the end $x_2 \in E(\S')$ and find a subsurface $S_2 \subset \S'$ that has a single boundary component $\gamma_2$, and $x_2 \in E(S_2)$ but $x_3 \not\in E(S_2)$. To simplify notation, let $\S'' \dfn \S \setminus \operatorname{int}(S_1 \cup S_2)$.  
By construction $\S''$ is a connected subsurface with two boundary components $\gamma_1$ and $\gamma_2$, and $x_3 \in E(\S'')$.   
Finally, let $\gamma_3$ be a separating simple closed curve that cuts $\S''$ into two components $K$ and $S_3$,
where $K$ is a compact subsurface with three boundary components $\gamma_1, \gamma_2,$ and $\gamma_3$, and $S_3$ is a subsurface with $E(S_3) = E(\S'')$. 
\Cref{pic:cptSubsurf} illustrates this setup.

\begin{figure}[ht!]
    \centering
    \begin{overpic}[width=7cm]{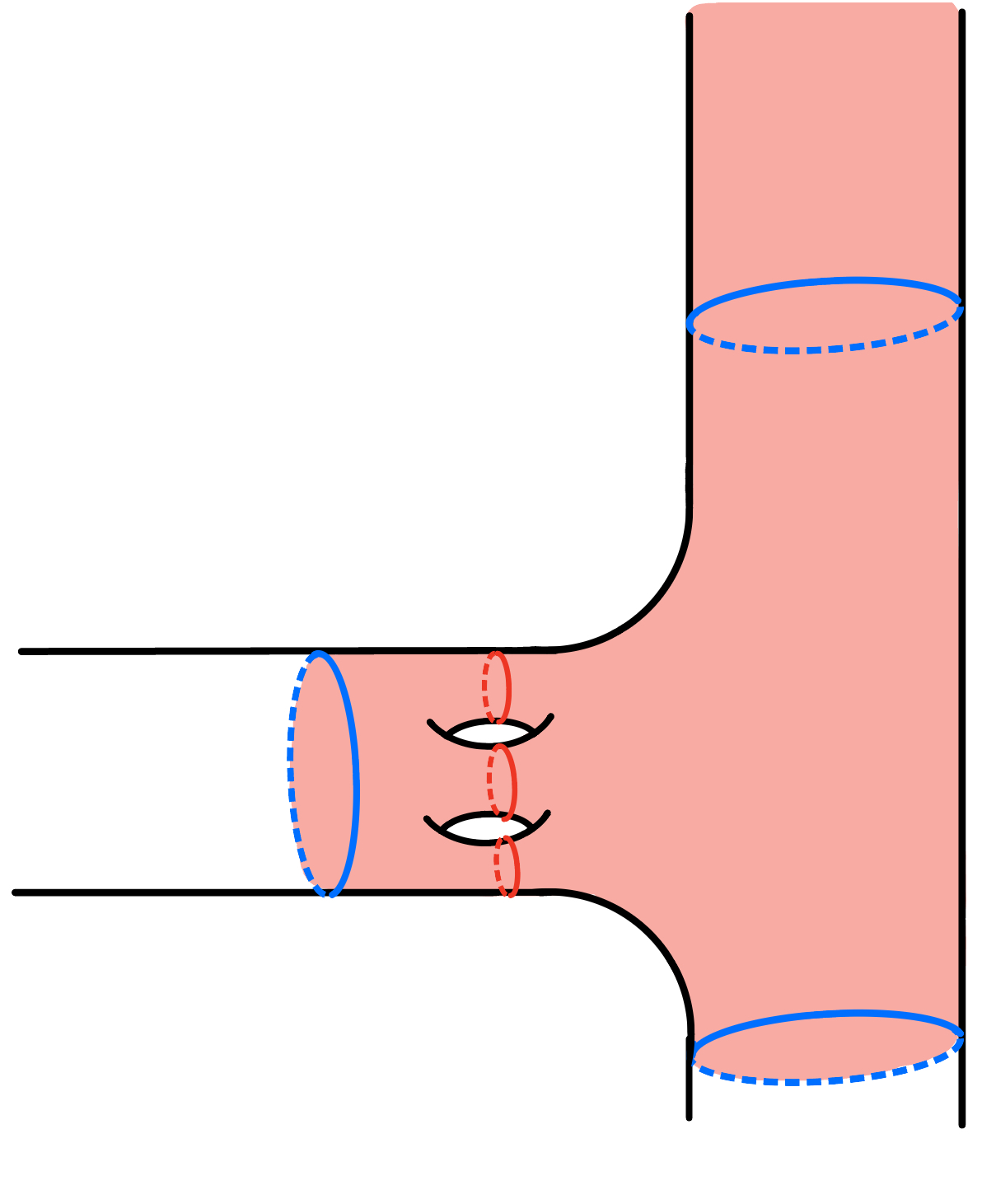}
    \put(18,28){$S_1$}
    \put(73,5){$S_2$}
    \put(73,80){$S_3$}
    \put(20,40){\color{blue} $\gamma_1$}
    \put(52,10){\color{blue} $\gamma_2$}
    \put(52,73){\color{blue} $\gamma_3$}
    \put(38,48){\color{red} $\partial S_1^i$}
    \put(25,5){$\underbrace{\hspace{5cm}}$}
    \put(53,-2){$\S'$}
    \put(65,45){\color{blue} $K$}
    \put(86,57){\color{red}$\S''$}
    \put(81,57){\color{red} $\begin{rcases*}{}\\{} \\{} \\{} \\{} \\{} \\{} \\{} \\{} \\{} \\{} \\{} \\{} \\{} \\ {} \end{rcases*}$}
    \end{overpic}
    \caption{The surface $\S$ partitioned into subsurfaces $S_1, S_2, S_3$ and $K$.  The subsurface $\S'' = \S \setminus \operatorname{int}(S_1 \cup S_2)$ is shaded in red.  
    \label{pic:cptSubsurf}
    }
\end{figure}
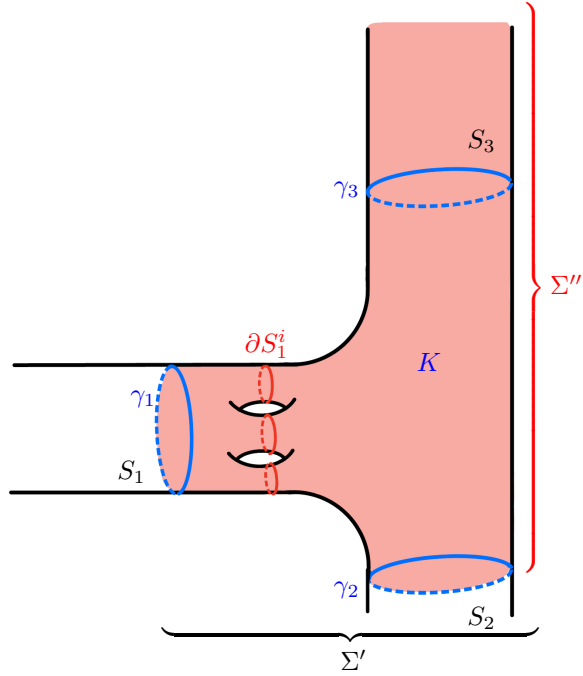

In summary, we have constructed a compact subsurface $K$ such that 
\begin{equation*} 
    \S \setminus \operatorname{int}(K) = S_1 \sqcup S_2 \sqcup S_3
\end{equation*} 
with $E(S_i) \ne \emptyset$ for each $i$, and $E(\S) = E(S_1) \sqcup E(S_2) \sqcup E(S_3)$. 

We now aim to prove that $K$ is nondisplaceable with respect to $\PMap(\S)$. Let $f \in \PMap(\S)$, and then 
$ 
\S \setminus \operatorname{int}(f(K)) = S_1' \sqcup S_2' \sqcup S_3'
$
where up to relabeling $S_i' = f(S_i)$.  We claim that 
\begin{equation}
\label{eqn:HomEnds}
E(S_i') = E(S_i). 
\end{equation}
By the relabeling, $E(S_i') = E(f(S_i))$. Since homeomorphisms $S_i \to S_i' $ induce homeomorphisms of the end spaces $E(S_i) \to E(S_i')$, it follows $E(f(S_i)) = f(E(S_i)).$
Finally, since $f$ is pure, $f$ fixes the ends point-wise, so $f(E(S_i)) = E(S_i).$  Putting this all together, 
$$E(S_i') = E(f(S_i)) = f(E(S_i)) = E(S_i)$$
which proves \Cref{eqn:HomEnds}.

Now, suppose that $K$ were displaceable.  Then $f(K) \cap K = \emptyset$, so $f(K)$ must be contained in one of the complementary components $S_i$.   Without loss of generality, assume $f(K) \subset S_1$. 
Because $S_2, S_3,$ and $K$ are all disjoint from $f(K)$ and $S_2 \cup S_3 \cup K$ is connected, the subsurface $S_2 \cup S_3 \cup K$ must be contained in one of the complementary components of $f(K)$. But this is not possible, since for example if $S_2 \cup S_3 \cup K \subset S_3'$, then $E(S_2) \subset E(S_3')$, but by \Cref{eqn:HomEnds}, $E(S_2) = E(S_2')$ contradicting $E(S_2') \cap E(S_3') = \emptyset$.  Similar arguments show $S_2 \cup S_3 \cup K \not\subset S_j'$ for $j = 1,2.$ 
    \end{proof}

Equipped with \Cref{prop:nondisp}, we are poised to begin the proof of \Cref{thm:mainthm}.  

\begin{namedtheorem}[\Cref{thm:mainthm} (a)]
    Let $\S$ be an infinite-type surface.  Then $\PMap(\S)$ is globally CB if and only if $\S$ is the Loch Ness monster surface. 
\end{namedtheorem}

\noindent\textit{Proof.}
By \Cref{prop:nondisp} and \Cref{lem:threeEnds}, if $|E(\S)| \ge 3$, then $\S$ has a nondisplaceable subsurface and therefore $\PMap(\S)$ is not CB.  
There are three infinite-type surfaces with $|E(\S)| \le 2$, namely the ladder surface, the once-punctured Loch Ness monster surface, and the Loch Ness monster surface, which we will consider case by case. 
\begin{itemize}[leftmargin=*]
\item \textbf{Case 1: $\S$ is the ladder surface.} For surfaces with $|E_G(\S)| \ge 2$, Aramayona, Patel and Vlamis \cite{aramayona2020first} define a \emph{flux map}, which is a continuous surjection $\PMap(\S) \twoheadrightarrow \Z$. Because the action of $\mathbb{Z}$ on $\mathbb{R}$ by translations is continuous and by isometries, but the orbits of $x \in \mathbb{R}$ are unbounded, $\mathbb{Z}$ is not CB. It follows that $\PMap(\S)$ is not CB either, by pulling back the action of $\mathbb{Z}$ on $\mathbb{R}$ to an action of $\PMap(\S)$ on $\mathbb{R}$.   

\item \textbf{Case 2: $\S$ is the once-punctured Loch Ness monster surface.} 
A subsurface containing the isolated puncture is nondisplaceable with respect to $\Mp(\S)$, and since $\S$ has two distinct ends of different type, $\Mp(\S) = \PMap(\S)$ in this case. Hence, $\S$ is not CB by \Cref{prop:nondisp}.

\item \textbf{Case 3: $\S$ is the Loch Ness monster surface.}
Because the Loch Ness monster has only one end, $\PMap(\S) = \Mp(\S)$ and, as before, the CB classification of $\Mp(\S)$ by Mann--Rafi applies.  Specifically, $|E(\S)| = 1$ means that the end space is trivially self-similar, so by \Cref{fact:MRgloballyCB}, $\Mp(\S)$ is CB.  \qed
\end{itemize}

\Cref{thm:mainthm} (a) is consistent with the results of Mann--Rafi.  In particular, when $|E(\S)| < \infty$, $\PMap(\S)$ has finite index in $\Mp(\S)$.  Because $\Mp(\S)$ is a Polish group and $\PMap(\S)$ is a clopen Polish subgroup, \Cref{fact:CEmbed} applies, and we 
can import Mann--Rafi's CB classification of $\Mp(\S)$ in \Cref{fact:MRCBgen} to $\PMap(\S)$.  As pointed out in \Cref{rmk:SSTelEndsInf}, the telescoping condition is only satisfied for surfaces with infinitely many ends, and self-similarity only occurs when there are infinitely many ends or exactly one end.  Therefore when $|E(\S)|< \infty$, the only surface with globally CB $\Mp(\S)$ is the unique infinite-type surface with one end, namely the Loch Ness Monster surface.

\subsection{Locally CB Classification}
We now turn our attention to classifying the surfaces for which $\PMap(\S)$ is locally CB.  For a finite subsurface $K \subset \Sigma$, let $$\mathcal{V}_K \dfn \set{f \in \PMap(S) \mid f_{{|K}} = \id}.$$ 
If $K = \emptyset$, then $\V_K = \PMap(\S)$.  

The proof of the proposition below follows the same argument as a result of Mann--Rafi for $\Mp(\S)$.

\begin{PROP}
[cf.  {\cite[Lemma 5.2]{mann2019large}}] 
    \label{lem:locallyCB}
    Let $K \subset \S$ be a finite-type (possibly empty) subsurface such that each component of $\S \setminus K$ is infinite-type. If $\S \setminus K$ contains a nondisplaceable subsurface with respect to $\PMap(\S)$, then $\V_K$ is not CB in $\PMap(\S)$.  If $\V_K$ is not CB in $\PMap(\S)$ for all such $K$, then $\PMap(\S)$ is not locally CB. 

\end{PROP}

\begin{proof}

First, let $S \subset \S \setminus K$ be a nondisplaceable subsurface with respect to $\PMap(\S)$.  
Recall in  \Cref{prop:nondisp} that $\PMap(\S)$ was shown to be non-CB when $S \subset \S$ is nondisplaceable by constructing an unbounded length function on $\PMap(\S)$.  
The same argument applies to construct an unbounded length function on $\V_K$. 
Specifically, consider a psuedo-Anosov $f$ supported on $S \subset \S \setminus K$ and extend $f$ to an element of $\PMap(\S)$ by extending via the identity on $\S \setminus S$.
Then $f \in \V_K$ and as in \Cref{prop:nondisp} the length function is unbounded on $\V_K$.

To see why $\PMap(\S)$ is not locally CB if all $\V_K$ are not CB in $\PMap(\S)$, first recall that $\set{\V_C \mid C \text{ compact}}$ is a neighborhood basis of the identity in the compact open topology.  For each compact subsurface $C$, there is a finite type subsurface $K$ that contains $C$ and whose complementary components $\S\setminus K$ are infinite-type. Specifically, take $K$ to be the union of $C$ and its finite-type complementary components.  So $\V_K \subset \V_C$.  It follows from the definition of CB that if $\V_K$ is not CB in $\PMap(\S)$, then $\V_C$ is not CB in $\PMap(\S)$. If all sets $\V_K$ are not CB in $\PMap(\S)$, then neither are any of the sets $\V_C$ in the neighborhood basis of the identity, and $\PMap(\S)$ has no CB neighborhood of the identity; in other words, $\PMap(\S)$ is not locally CB.  
\end{proof}

We will see that \Cref{lem:locallyCB} effectively narrows down the possible surfaces with a locally CB pure mapping class group to those with finitely many ends.  

\begin{namedtheorem}[\Cref{thm:mainthm} (b)]
Let $\S$ be an infinite-type surface.  Then $\PMap(\S)$ is locally CB if and only if $|E(\S)| < \infty$ and either 
    \begin{itemize}
        \item[\emph{(i)}] $|E_G(\S)| = 1$ and $|E(\S) \setminus E_G(\S)| = 0$ or, 
        \item[\emph{(ii)}] $|E_G(\S)| > 1$ and $|E(\S) \setminus E_G(\S)| \ge 0$;
    \end{itemize}  
    in other words, $\S$ is not a Loch Ness monster surface with (nonzero) punctures. 
\end{namedtheorem}

\begin{proof}

First, consider when $|E(\S)| = \infty$. 
Let $K \subset \S$ be a finite-type subsurface (possibly empty) so that each connected component of $\S \setminus K$ has infinite-type.  
If $K = \emptyset$, then $\V_K = \PMap(\S)$ is not CB by \Cref{thm:mainthm} (a).  For any $K \ne \emptyset$, at least one of the connected components of $\S \setminus K$ must contain infinitely many ends, and by \Cref{lem:threeEnds} it must contain a nondisplaceable subsurface.  So, $\V_K$ is not CB in $\PMap(\S)$ for all possible $K$ and by  \Cref{lem:locallyCB}, $\PMap(\S)$ is not locally CB. 

Before proceeding, we remark that when $|E(\S)| < \infty$, we can determine if $\PMap(\S)$ is locally CB solely by establishing whether $\Mp(\S)$ is locally CB. 
To see why, let $\V$ be a CB neighborhood of the identity in $\Mp(\S)$.  The restriction of $\V$ to $\PMap(\S)$ is CB in $\PMap(\S)$ by \Cref{fact:CEmbed} since $|E(\S)|< \infty$. Furthermore, $\V$ restricts to an identity neighborhood in $\PMap(\S)$, so $\PMap(\S)$ is locally CB. 

Now when $|E_G(\S)| = 1$, and $|E(\S) \setminus E_G(\S)| = 0$, the surface $\S$ is the Loch Ness monster surface which is globally CB by \Cref{thm:mainthm} (a) and hence locally CB. On the other hand, when $|E_G(\S)| = 1$ and $\S$ has nonzero, but finitely many punctures, we can show that $\Mp(\S)$ is not locally CB by \Cref{fact:MRlocallyCB}.  Up to isotopy, there is only subsurface $K$ satisfying the hypothesis of \Cref{fact:MRlocallyCB}, namely the subsurface of genus 0 containing all finitely many punctures pictured in  \Cref{pic:finpunctures}.

\begin{figure}[ht!]
    
    \centering
    \begin{overpic}[width=5cm]{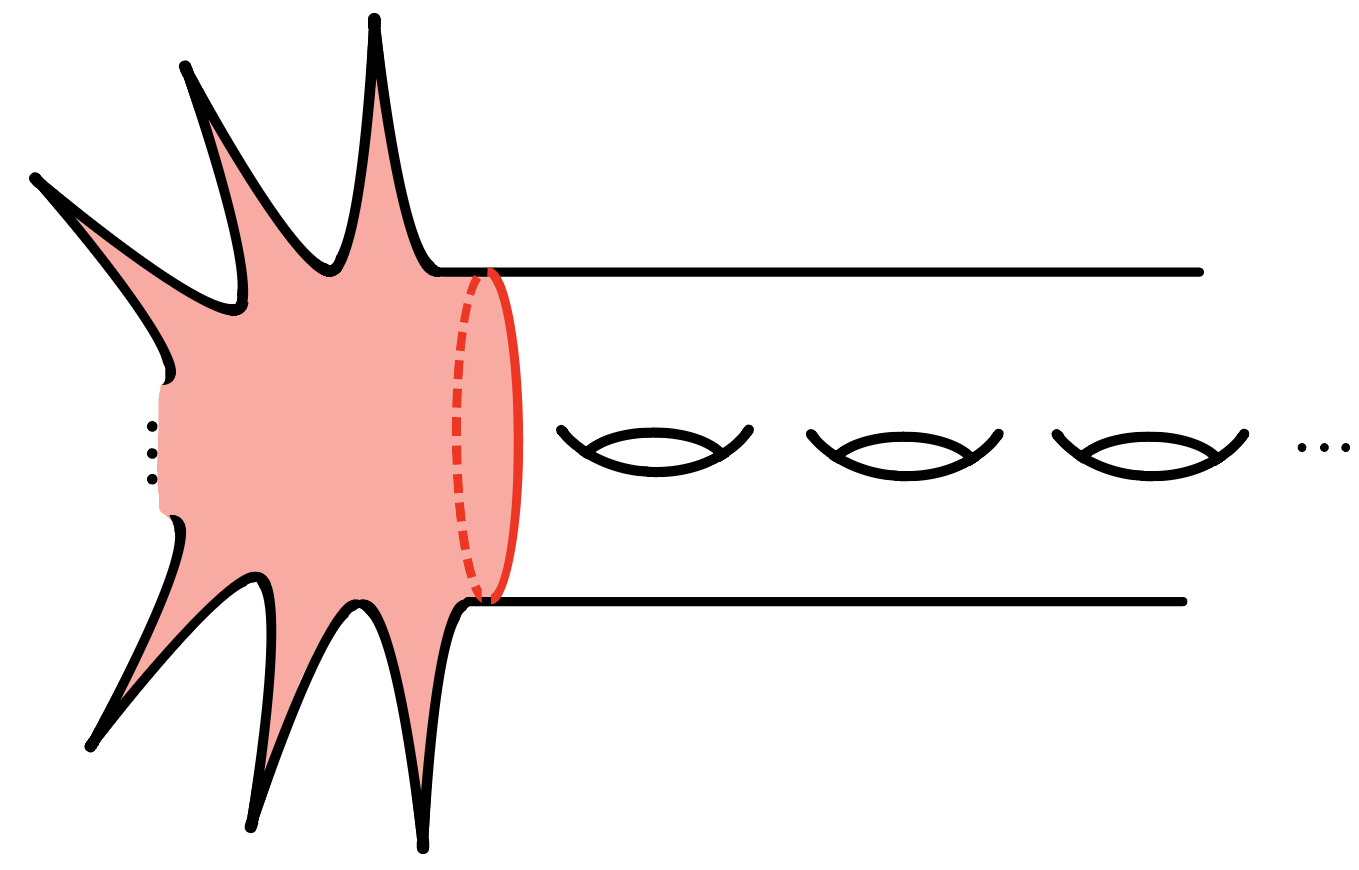}
    \put(20,30){\color{red} $K$}
    \end{overpic}
    \caption{The unique subsurface $K$ of the finitely punctured Loch Ness monster surface satisfying the hypothesis of \Cref{fact:MRlocallyCB} up to isotopy. 
    \label{pic:finpunctures}
    }
\end{figure}

Let $V$ be a neighborhood of the end accumulated by genus.  In general, $\S \setminus V$ has some nonzero finite genus, while $\S \setminus (\S \setminus K) = K$ has zero genus. So by the classification of surfaces, there is no $f \in \Mp(\S)$ such that $f(V) \supset \S \setminus K$, and so $\Mp(\S)$ is not locally CB.  
Hence, when $|E_G(\S)| = 1$ and $|E(\S) \setminus E_G(\S)| \ne 0$, $\PMap(\S)$ is not locally CB.  

Finally, we will show $\PMap(\S)$ is locally CB when $|E(\S)| < \infty$ and $|E_G(\S)| > 1$, by showing that $\Mp(\S)$ is locally CB in this case. Because $\S$ has finitely many ends, it has, at most, a finite number of punctures.  Let $K \subset \S$ be a finite-type subsurface containing all punctures of $\S$ (if any), such that all components of $\S \setminus K$ contain a single end accumulated by genus.  These self-similar singleton sets represent the $A$ sets in \Cref{fact:MRlocallyCB}. There are no $P$ sets for this surface. See \Cref{pic:finPunctMultEnds}. 

\begin{figure}[ht!]
    
    \centering
   \begin{overpic}[width=7cm]{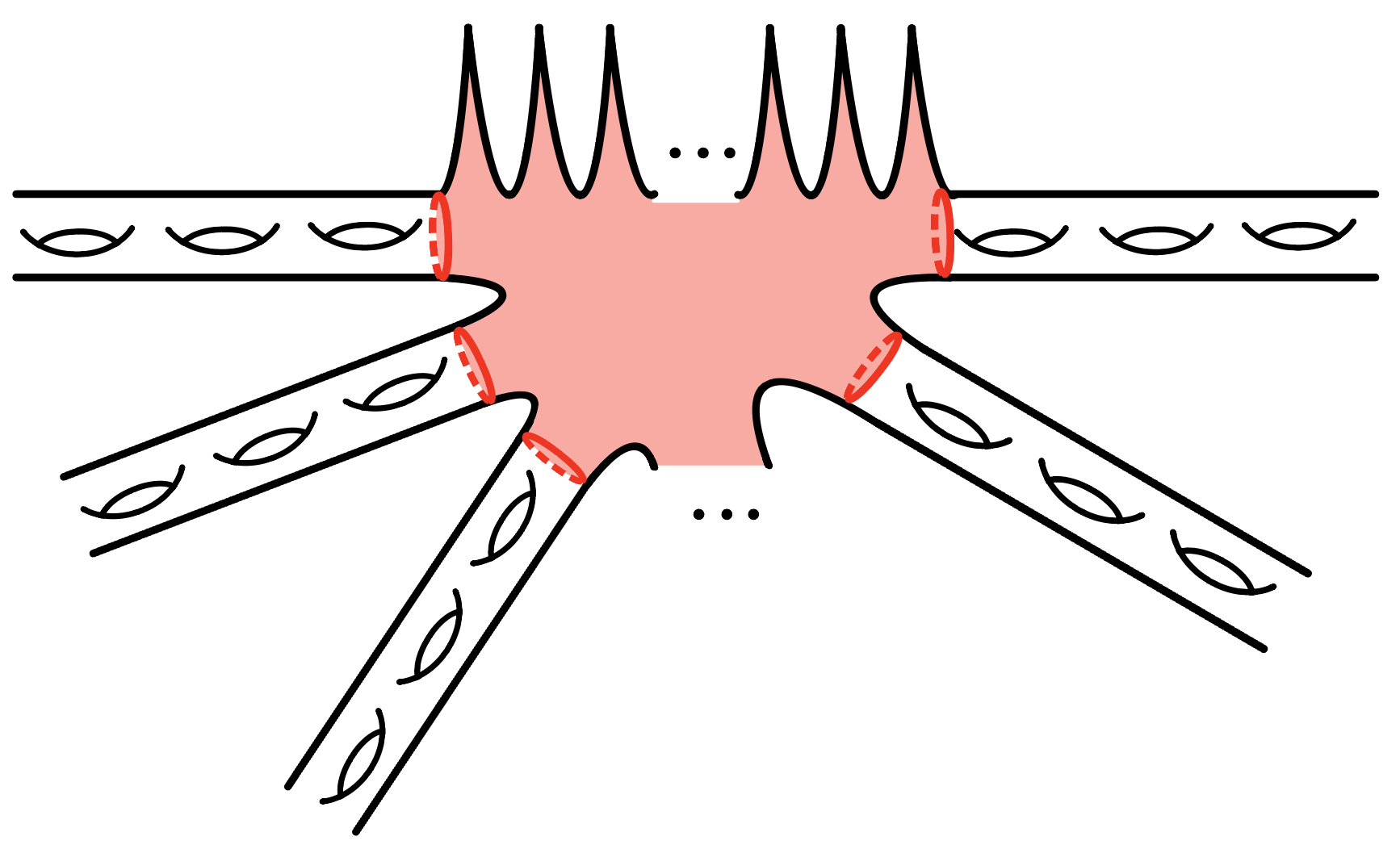}
        \put(47,38){\color{red} $K$}
   \end{overpic}
    \caption{A choice of $K$ that satisfies the hypothesis of \Cref{fact:MRlocallyCB}. 
    \label{pic:finPunctMultEnds}
    }
\end{figure}

We now check condition (3) in the statement of \Cref{fact:MRlocallyCB}. For any end accumulated by genus $x \in E_G(\S)$ and any neighborhood $V$ of $x$, an iterate of a handle shift $h \in \Mp(\S)$ between $x$ and some other end $y \ne x \in E_G(\S)$ will shift all of the genus in the component of $\S \setminus K$ not in $V$ into the component of $\S \setminus K$ containing $y$.  Thus, $h$ maps the neighborhood $V$ to the entire complementary component of $\S \setminus K$ containing the end $x$, and by \Cref{fact:MRlocallyCB} $\Mp(\S)$ is locally CB, hence $\PMap(\S)$ is as well.   
\end{proof}

Finally, we show that the surfaces with CB-generated pure mapping class group coincide with those with locally CB pure mapping class group, finishing the proof of \Cref{thm:mainthm}. 

\begin{namedtheorem}[\Cref{thm:mainthm} (c)]
    Let $\S$ be an infinite-type surface.  Then $\PMap(\S)$ is CB-generated if and only if $\PMap(\S)$ is locally CB. 
\end{namedtheorem}

\begin{proof}

    Recall that every CB-generated Polish group is locally CB {\cite[Theorem 1.2]{rosendal2021coarse}}, so we need only prove that if $\PMap(\S)$ is locally CB, then it is CB-generated. In the case where $\S$ is the Loch Ness Monster, \Cref{thm:mainthm}(a) implies that $\PMap(\S)$ is globally CB and hence CB-generated. Thus, it remains to consider the case where $\S$ has finitely many ends and $|E_G(\S)|> 1$.

    Since $\PMap(\S)$ is locally CB by \Cref{thm:mainthm} (b), we know that for $K\subset \S$ as in \Cref{fact:MRlocallyCB}, $\V_K$ is a CB neighborhood of the identity in $\PMap(\S)$.
    Now, let $D$ be a Lickorish generating set for the finite-type subsurface $K$, and $H$ a finite collection of pairwise commuting handle-shifts and chosen by Mann--Rafi in the proof of \Cref{fact:MRCBgen}. As both $D$ and $H$ are finite, the set $X = \V_K \cup D \cup H$ is CB in $\PMap(\S)$.  We claim that $X$ generates $\PMap(S)$. In \cite[Lemma 6.21]{mann2019large}, they show that the set $X$ together with a collection of ``generalized" shift maps generates a group $G$ that contains $\PMap(\S)$. However, in our case, $|E(\S)|< \infty$, so there are no generalized shift maps. Since all elements of $X$ are pure, we conclude that the CB set $X$ generates $\PMap(\S)$.
\end{proof}

\bibliography{refs}
\bibliographystyle{alpha}

\end{document}